\documentclass[12pt]{amsart}
\usepackage{amsfonts,amssymb,amscd,amsmath,enumerate,verbatim}
\usepackage[latin1]{inputenc}
\usepackage{amscd}
\usepackage{latexsym}
\usepackage{mathptmx}
\usepackage{multicol}
\input xy
\xyoption{all}
\def\NZQ{\mathbb}               % the font for N,Z,Q,R,C

\def\ZZ{{\NZQ Z}}
\def\RR{{\NZQ R}}
\def\CC{{\NZQ C}}

\def\frk{\mathfrak}               % font for "Fraktur"

\def\Phi{{\frk N}}
\def\ab{{\bold a}}
\def\bb{{\bold b}}

\def\eb{{\bold e}}

\def\vb{{\bold v}}

\def\xb{{\bold x}}
\def\yb{{\bold y}}

\def\opn#1#2{\def#1{\operatorname{#2}}} % to make operators
\opn\chara{char} 
\opn\length{\ell} 
\opn\pd{pd} 
\opn\rk{rk}
\opn\projdim{proj\,dim} 
\opn\injdim{inj\,dim} 
\opn\rank{rank}
\opn\depth{depth} 
\opn\grade{grade} 
\opn\height{height}
\opn\embdim{emb\,dim} 
\opn\codim{codim}

\opn\Tr{Tr} 
\opn\bigrank{big\,rank}
\opn\superheight{superheight}
\opn\lcm{lcm}
\opn\trdeg{tr\,deg}%\emph{
\opn\reg{reg} 
\opn\lreg{lreg} 
\opn\ini{in} 
\opn\lpd{lpd}
\opn\size{size}
\opn\mult{mult}
\opn\dist{dist}
\opn\cone{cone}
\opn\lex{lex}
\opn\rev{rev}
\opn\div{div} \opn\Div{Div} \opn\cl{cl} \opn\Cl{Cl}
\opn\Spec{Spec} \opn\Supp{Supp} \opn\supp{supp} \opn\Sing{Sing}
\opn\Ass{Ass} \opn\Min{Min}
\opn\Ann{Ann} \opn\Rad{Rad} \opn\Soc{Soc}
\opn\Syz{Syz} \opn\Im{Im} \opn\Ker{Ker} \opn\Coker{Coker}
\opn\Am{Am} \opn\Hom{Hom} \opn\Tor{Tor} \opn\Ext{Ext}
\opn\End{End} \opn\Aut{Aut} \opn\id{id} \opn\ini{in}

\opn\nat{nat}
\opn\pff{pf}%   \pf exists already
\opn\Pf{Pf} \opn\GL{GL} \opn\SL{SL} \opn\mod{mod} \opn\ord{ord}
\opn\Gin{Gin}
\opn\Hilb{Hilb}\opn\adeg{adeg}\opn\std{std}\opn\ip{infpt}
\opn\Pol{Pol}
\opn\sat{sat}
\opn\Var{Var}
\opn\Gen{Gen}

\opn\aff{aff} \opn\con{conv} \opn\relint{relint} \opn\st{st}
\opn\lk{lk} \opn\cn{cn} \opn\core{core} \opn\vol{vol}
\opn\link{link} \opn\star{star}
\opn\gr{gr}

\def\Pc{{\mathcal P}}

\def\vol{{\textnormal{vol}}}
\def\hei{{\textnormal{ht}}}
\def\ord{{\textnormal{ord}}}
\def\pot#1#2{#1[\kern-0.28ex[#2]\kern-0.28ex]}

\opn\dirlim{\underrightarrow{\lim}}
\opn\inivlim{\underleftarrow{\lim}}
\def\Implies{\ifmmode\Longrightarrow \else
	\unskip${}\Longrightarrow{}$\ignorespaces\fi}
\def\implies{\ifmmode\Rightarrow \else
	\unskip${}\Rightarrow{}$\ignorespaces\fi}
\def\iff{\ifmmode\Longleftrightarrow \else
	\unskip${}\Longleftrightarrow{}$\ignorespaces\fi}

\let\:=\colon
\newtheorem{Theorem}{Theorem}[section]
\newtheorem{Lemma}[Theorem]{Lemma}
\newtheorem{Corollary}[Theorem]{Corollary}
\newtheorem{Proposition}[Theorem]{Proposition}
\newtheorem{Remark}[Theorem]{Remark}

\newtheorem{Example}[Theorem]{Example}

\newtheorem{Problem}[Theorem]{Problem}
\newtheorem{Conjecture}[Theorem]{Conjecture}

\let\epsilon\varepsilon
\let\phi=\varphi
\let\kappa=\varkappa
\def\qed{\ifhmode\textqed\fi
	\ifmmode\ifinner\quad\qedsymbol\else\dispqed\fi\fi}
\def\textqed{\unskip\nobreak\penalty50
	\hskip2em\hbox{}\nobreak\hfil\qedsymbol
	\parfillskip=0pt \finalhyphendemerits=0}
\def\dispqed{\rlap{\qquad\qedsymbol}}

\opn\dis{dis}
\opn\height{height}
\opn\dist{dist}
\def\pnt{{\raise0.5mm\hbox{\large\bf.}}}

\opn\Lex{Lex}
\opn\conv{conv}

\begin{document}
\title{Gorenstein simplices with a given $\delta$-polynomial}
\author[T.~Hibi]{Takayuki Hibi}
\address[Takayuki Hibi]{Department of Pure and Applied Mathematics,
	Graduate School of Information Science and Technology,
	Osaka University,
	Suita, Osaka 565-0871, Japan}
\email{hibi@math.sci.osaka-u.ac.jp}

\author[A.~Tsuchiya]{Akiyoshi Tsuchiya}
\address[Akiyoshi Tsuchiya]{Akiyoshi Tsuchiya,
Graduate school of Mathematical Sciences,
University of Tokyo,
Komaba, Meguro-ku, Tokyo 153-8914, Japan} 
\email{akiyoshi@ms.u-tokyo.ac.jp}

\author[K.~Yoshida]{Koutarou Yoshida}
\address[Koutarou Yoshida]{Department of Pure and Applied Mathematics,
	Graduate School of Information Science and Technology,
	Osaka University,
	Suita, Osaka 565-0871, Japan}
\email{kt-yoshida@ist.osaka-u.ac.jp}
\subjclass[2010]{52B12, 52B20}
%\date{}
\keywords{Lattice polytope, Gorenstein polytope, $\delta$-polynomial, empty simplex}
\begin{abstract}
To classify the lattice polytopes 
with a given $\delta$-polynomial is an important open problem in Ehrhart theory. 
A complete classification of the Gorenstein simplices whose normalized volumes are prime integers is known.
In particular, their   
$\delta$-polynomials are 
of the form $1+t^k+\cdots+t^{(v-1)k}$, 
where $k$ and $v$ are positive integers.
In the present paper, a complete classification of the Gorenstein simplices with the above $\delta$-polynomials will be performed,
when $v$ is either $p^2$ or $pq$, where $p$ and $q$ are prime integers 
with $p \neq q$. Moreover, we consider  the number of Gorenstein simplices, up to unimodular equivalence, 
with the expected $\delta$-polynomial.
\end{abstract} 

\maketitle
\section*{Introduction}
To classify the lattice polytopes 
with a given $\delta$-polynomial is an important open problem among the study on lattice polytopes.
A {\em lattice polytope} is a convex polytope
$\Pc \subset \RR^d$ all of whose vertices have integer coordinates.
Recall from \cite{CCD} and \cite[Part II]{HibiRedBook}
what the {\em $\delta$-polynomial} of $\Pc$ is. 

Let $\Pc \subset \RR^d$ be a lattice polytope of dimension $d$
and define $\delta(\Pc, t)$ by the formula
\[
(1 - t)^{d+1} \delta(\Pc, t)
= 1 + \sum_{n=1}^{\infty} |n\Pc \cap \ZZ^d| t^n,
\]
where $n\Pc=\{n \ab : \ab \in \Pc \}$, the $n$th dilated 
polytopes of $\Pc$.
It follows that $\delta(\Pc, t)$ is a polynomial in $t$ 
of degree at most $d$.
We say that $\delta(\Pc, t)$
is the {\em $\delta$-polynomial} or {\em $h^*$-polynomial}  of $\Pc$.
Let $\delta(\Pc, t) = \delta_0+\delta_1 t+\cdots+\delta_d t^d$.
Then $\delta_0 = 1$, $\delta_1= |\Pc \cap \ZZ^d|-(d+1)$ and $\delta_d = |(\Pc \setminus \partial \Pc) 
\cap \ZZ^d|$, where $\partial \Pc$ is the boundary of $\Pc$,
and each $\delta_i \geq 0$.
When $\delta_d \neq 0$, one has $\delta_i \geq \delta_1$
for $1 \leq i \leq d$.
Moreover, $\delta(\Pc,1) = \sum_{i=0}^{d}\delta_i$ 
coincides with the {\em normalized volume} 
$\text{Vol}(\Pc)$ of $\Pc$.  

A lattice polytope $\Pc \subset \RR^d$ of dimension $d$
is called {\em reflexive} if the origin of $\RR^d$ belongs to 
the interior of $\Pc$ and the dual polytope
(\cite[pp.~103--104]{HibiRedBook}) of $\Pc$ is again 
a lattice polytope.  A lattice polytope $\Pc \subset \RR^d$ 
of dimension $d$ is called {\em Gorenstein of index} $r$ if $r\Pc$ is unimodularly equivalent to a reflexive polytope. 
It is known that 
$\Pc \subset \RR^d$ is Gorenstein if and only if the $\delta$-polynomial
$\delta(\Pc,t)=\delta_0+\delta_1 t+\cdots+\delta_s t^s$,
where $\delta_s \neq 0$ is palindromic, i.e.,
$\delta_i=\delta_{s-i}$ for each $0 \leq i \leq \lfloor s/2\rfloor$.

Gorenstein polytopes are of interest in commutative algebra, 
mirror symmetry and tropical geometry (\cite{Batyrev,JK}).
In each dimension, there exist only finitely many Gorenstein polytopes 
up to unimodular equivalence (\cite{Lag})
and, in addition, Gorenstein polytopes are completely classified
up to dimension $4$ (\cite{Kre}).
Recently certain classification results of higher-dimensional 
Gorenstein polytopes are obtained by \cite{BJ,HNT,Tsuchiya}. 

The final goal of one of our research projects is to classify 
the Gorenstein simplices with given $\delta$-polynomials.
Classification problems of lattice simplices have been studied by many authors and there are several related works (e.g., see \cite{BDS,HNO}).
In \cite[Corollary 2.4]{Tsuchiya} it is shown that if
$\Delta$ is a Gorenstein simplex whose normalized volume 
$\text{Vol}(\Delta)$ is a prime integer $p$, then its 
$\delta$-polynomial is of the form 
\[
\delta(\Delta,t)=1+t^{k}+\cdots+t^{(p-1)k},
\]
where $k$ is a positive integer (Proposition \ref{moti}).
In the present paper, from the fact we focus on the following problem:
\begin{Problem}
\label{pro}
	Given positive integers $k$ and $v$,
	classify the Gorenstein simplices with the $\delta$-polynomial 
	$1+t^{k}+\cdots+t^{(v-1)k}$.
\end{Problem}

A lattice simplex is called {\em empty} 
if it possesses no lattice point except for its vertices. 
A lattice simplex $\Delta$ 
with $\delta(\Delta,t)=\delta_0+\delta_1 t+\cdots+\delta_d t^d$
is empty if and only if $\delta_1 = 0$.
In particular, in Problem \ref{pro}, when $k > 1$, 
its target is Gorenstein empty simplices.

The present paper is organized as follows:
Section $1$ consists of the review of fundamental materials 
on lattice simplices and the collection of indispensable lemmata.
We devote Section $2$ to discuss a lower bound on the dimensions 
of Gorenstein simplices with a given $\delta$-polynomial 
of Problem \ref{pro} and, in addition, to classify 
the Gorenstein simplices when the lower bound holds (Proposition \ref{inf}).
The highlight of this paper is Section $3$, where
a complete answer of Problem \ref{pro} will be given when $v$ is 
either $p^2$ or $pq$, where $p$ and $q$ are distinct prime integers (Theorems \ref{main1} and \ref{main2}).  
Finally, in Section $4$,
we will discuss 
the number of Gorenstein simplices, up to unimodular equivalence, 
with a given $\delta$-polynomial of Problem \ref{pro}.

\section*{Acknowledgment}
The authors would like to thank anonymous referees for reading the manuscript carefully.
The second author was partially supported by Grant-in-Aid for JSPS Fellows 16J01549.

\section{Preliminaries}
In this section, we recall basic materials on lattice simplices and we prepare the essential lemmata in this paper.

At first, we introduce the associated finite abelian groups of lattice simplices.
For a lattice simplex $\Delta \subset \RR^d$ of dimension $d$ whose vertices are $\vb_0,\ldots,\vb_d \in \ZZ^d$
	set 
	$$\Lambda_\Delta=\{(\lambda_0,\ldots,\lambda_d) \in (\RR/\ZZ)^{d+1} : \sum\limits_{i=0}^{d}\lambda_i(\vb_i,1) \in \ZZ^{d+1}   \}.$$
	The collection $\Lambda_\Delta$ forms a finite abelian group with addition defined as follows: 
	For $(\lambda_0,\ldots,\lambda_d) \in (\RR/\ZZ)^{d+1}$ and $(\lambda_0',\ldots,\lambda_d') \in (\RR/\ZZ)^{d+1}$,  $(\lambda_0,\ldots,\lambda_d)+(\lambda_0',\ldots,\lambda_d')=(\lambda_0+\lambda_0',\ldots,\lambda_d+\lambda_d') \in (\RR/\ZZ)^{d+1}$.
We denote the unit of $\Lambda_\Delta$ by $\bold{0}$, and the inverse of $\bold{x}$ by $-\bold{x}$,
and also denote $\underbrace{\bold{x}+\cdots+\bold{x}}_{j}$  by $j\bold{x}$ for an integer $j>0$ and $\bold{x} \in \Lambda_\Delta$.
	For $\xb=(x_0,\ldots,x_{d}) \in \Lambda_{\Delta}$, where each $x_i$ is taken with $0 \leq x_i <1$, we set
	$\textnormal{ht}(\xb)=\sum_{i=0}^{d}x_i \in \ZZ$
	and $\textnormal{ord}(\xb)=\min\{\ell \in \ZZ_{> 0} : \ell\xb={\bf 0} \}$.

It is well known that the $\delta$-polynomial of the lattice simplex $\Delta$ can be computed as follows: 
\begin{Lemma}
	\label{delta}
Let $\Delta$ be a lattice simplex of dimension $d$ whose $\delta$-polynomial  equals $1+\delta_1 t+\cdots+\delta_{d}t^d$.
		Then for each $i$, we have $\delta_i=\sharp\{\lambda \in \Lambda_\Delta : \textnormal{ht}(\lambda)=i\}$.
\end{Lemma} 
  
Recall that a matrix $A \in \ZZ^{d\times d}$ is $unimodular$ if det($A$)$=\pm 1$.
For lattice polytopes $\mathcal{P},\mathcal{Q} \subset \RR^d$ of dimension $d$, $\mathcal{P}$ and $\mathcal{Q}$ are called {\em unimodularly equivalent} if there exist a unimodular matrix $U \in \ZZ^{d\times d}$ and an integral vector $\bold{w} \in \ZZ^d$ such that $\mathcal{Q}=f_U(\mathcal{P})+\bold{w}$, where $f_U$ is the linear transformation in $\RR^d$ defined by $U$, i.e., $f_U(\bold{v})=\bold{v}U$ for all $\bold{v} \in \RR^d$.

	In \cite{BH}, it is shown that there is a bijection between unimodular equivalence classes of $d$-dimensional lattice simplices with a 
	chosen ordering of their vertices and finite subgroups of $(\RR/\ZZ)^{d+1}$ such that the sum of all entries of each element is an integer.
	In particular, two lattice simplices $\Delta$ and $\Delta'$ are unimodularly equivalent if and only if there exist  orderings of their vertices such that $\Lambda_\Delta=\Lambda_{\Delta'}$.

For a lattice polytope $\mathcal{P} \subset \RR^d$ of dimension $d$, the {\em lattice pyramid} over $\mathcal{P}$ is defined by $\text{conv}(\mathcal{P}\times \left\{ 0 \right\} ,(0,\ldots,0,1))$ $\subset \RR^{d+1}$. Let $\text{Pyr}(\mathcal{P})$ denote this polytope. We use the term lattice pyramid for a lattice polytope that has been obtained by successively taking lattice pyramids.
A characterization of lattice pyramids in terms of the associated finite abelian groups is known.
		\begin{Lemma}[{\cite[Lemma 12]{Nill}}]
		Let $\Delta \subset \RR^d$ be a lattice simplex of dimension $d$.
		Then $\Delta$ is a lattice pyramid if and only if there is $i \in \{0,\ldots,d\}$ such that $\lambda_i=0$ for all $(\lambda_0,\ldots,\lambda_d) \in \Lambda_\Delta$.
	\end{Lemma}

For a lattice polytope $\mathcal{P} \subset \RR^d$ of dimension $d$, one has $\delta(\Pc,t)=\delta(\text{Pyr}(\mathcal{P}),t)$.
Therefore, it is essential that we characterize polytopes which are not lattice pyramids  over any lower-dimensional lattice simplex.

Finally, we give some lemmata. 
These lemmata are characterizations of some Gorenstein simplices in terms of the associated finite abelian groups. 

\begin{Lemma}[{\cite[Theorem 3.2]{Tsuchiya}}]
	\label{p^2}
	Let $p$ be a prime integer and $\Delta \subset \RR^d$ a $d$-dimensional lattice simplex whose normalized volume equals $p^2$.
	Suppose that $\Delta$ is not a lattice pyramid over any lower-dimensional lattice simplex.
	Then $\Delta$ is Gorenstein of index $r$ if and only if one of the followings is satisfied:
	\begin{enumerate}
		\item There exists an integer $s$ with $0 \leq s \leq d-1$ such that $rp^2-1=(d-s)+ps$ and $\Lambda_{\Delta}$ 
is generated by $\left(\underbrace{1/p,\ldots,1/p}_{s},\underbrace{1/p^2,\ldots,1/p^2}_{d-s+1}\right)$ 
		for some ordering of the vertices of $\Delta$, or
		\item  $d=rp-1$ and there exist integers $0 \leq a_0,\ldots,a_{d-2} \leq p-1$
		with $p \mid (a_0+\cdots+a_{d-2}-1)$
		such that  $\Lambda_{\Delta}$ is generated by 
$\left((a_0+1)/p,\ldots,(a_{d-2}+1)/p,0,1/p \right)$ 
and $\left((p-a_0)/p,\ldots,(p-a_{d-2})/p,1/p,0 \right)$
for some ordering of the vertices of $\Delta$.
	\end{enumerate}
\end{Lemma}

	\begin{Lemma}[{\cite[Theorem 3.3]{Tsuchiya}}]
		\label{pq}
	Let $p$ and $q$ be prime integers with $p \neq q$ and 
	$\Delta \subset \RR^d$ a $d$-dimensional lattice simplex whose normalized volume equals $pq$.
	Suppose that $\Delta$ is not a lattice pyramid over any lower-dimensional lattice simplex.
	Then $\Delta$ is Gorenstein of index $r$ if and only if 
	there exist nonnegative integers $s_1,s_2,s_3$ with $s_1+s_2+s_3=d+1$ such that the following conditions are satisfied:
	\begin{enumerate}
		\item $rpq=s_1q+s_2p+s_3$, and
		\item$\Lambda_{\Delta}$ is generated by
		$\left(
		\underbrace{1/p,\ldots,1/p}_{s_1},\underbrace{1/q,\ldots,1/q}_{s_2},\underbrace{1/(pq),\ldots,1/(pq)}_{s_3}\right)$ 	for some ordering of the vertices of $\Delta$. 
	\end{enumerate}
\end{Lemma}

\section{Existence}  
 In this section, we prove that for positive integers $k$ and $v$, there exists a lattice simplex with the $\delta$-polynomial $1+t^{k}+t^{2k}+\cdots+t^{(v-1)k}$. 
Moreover, we give a lower bound and an upper bound on the dimension of such a lattice simplex which is not a lattice pyramid. 
In fact, we obtain the following proposition.
\begin{Proposition}
	\label{inf}
		Let $v$ and $k$ be positive integers.
		Then there exists a Gorenstein simplex $\Delta \subset \RR^d$ of dimension $d$ whose $\delta$-polynomial is 
		$1+t^{k}+t^{2k}+\cdots+t^{(v-1)k}$.
Furthermore, if $\Delta$ is not a lattice pyramid over any lower-dimensional lattice simplex, 
then one has $vk-1 \leq d \leq 4(v-1)k-2$.
In particular, the lower bound holds if and only if  $\Lambda_{\Delta}$ is generated by $(1/v,\ldots,1/v)$.
	\end{Proposition}
\begin{proof}
	Let $\Delta_{v,k}$ be a lattice simplex such that $\Lambda_{\Delta_{v,k}}$ is generated by $(1/v,\ldots,1/v)$, where $d=vk-1$.
	Then it follows from Lemma \ref{delta}  that $\delta(\Delta_{v,k},t)=1+t^k+t^{2k}+\cdots+t^{(v-1)k}$.

Now,  let $\Delta \subset \RR^d$ be a lattice simplex of dimension $d$ whose 
$\delta$-polynomial is 
$1+t^{k}+t^{2k}+\cdots+t^{(v-1)k}$.
Let $\bold{x}=(x_0,\ldots,x_d) \in \Lambda_{\Delta}$ be an element such that $\textnormal{ht}(\xb)=(v-1)k$.
Then we have that $\textnormal{ht}(-\xb) \geq k$. 
Hence since $\textnormal{ht}(\xb)+ \textnormal{ht}(-\xb) \leq d+1$, we obtain $d \geq  vk-1$.
From {\cite[Theorem 7]{Nill}}, if $\Delta$ is not a lattice pyramid over any lower-dimensional lattice simplex,
then one has  $d \leq 4(v-1)k-2$.
Now, we assume that $d=vk-1$.
Since for each $i$, one has $0 \leq x_i \leq (v-1)/v$,
we obtain $\hei(\xb) \leq (d+1)(v-1)/v=(v-1)k$.
Hence for each $i$, it follows that $x_i=(v-1)/v$.
Therefore  $\Lambda_{\Delta}$ is generated by $(1/v,\ldots,1/v)$.
Then it is easy to show that
$\delta(\Delta,t)=1+t^k+t^{2k}+\cdots+t^{(v-1)k}$,
as desired.  
\end{proof}

\section{Classification}
In this section, we give a complete answer of Problem \ref{pro} for the case that $v$ is the product of two prime integers. 
First, we consider the case where $v$ is a prime integer.
The following proposition motivates us to consider Problem \ref{pro}. 

\begin{Proposition}[{\cite[Corollary 2.4]{Tsuchiya}}]
	\label{moti}
	Let $p$ be a prime integer and $\Delta \subset \RR^d $ a Gorenstein simplex of index $r$ whose normalized volume equals $p$.
	Suppose that $\Delta$ is not a lattice pyramid over any lower-dimensional lattice simplex.
	Then $d=rp-1$ and $\Lambda_{\Delta}$ is generated by $(1/p,\ldots,1/p)$.
	Furthermore, one has 
	$\delta(\Delta,t)=1+t^{r}+t^{2r}+\cdots+t^{(p-1)r}$.
\end{Proposition}

This theorem says that for any positive integers $k$ and $v$, if $v$ is a prime integer, then there exists just one lattice simplex up to unimodular equivalence such that its $\delta$-polynomial equals $1+t^k+t^{2k}+\cdots+t^{(v-1)k}$.
By the following proposition, we know that if $v$ is not a prime integer, then there exist at least two such simplices up to unimodular equivalence.

\begin{Proposition}\label{prop1}
	Given positive integers $k$, $v$ and a proper divisor $u$ of $v$, let $\Delta \subset \RR^d$ be a lattice simplex of dimension $d$ such that  
	$\Lambda_{\Delta}$ is generated by 
	$$\left(\underbrace{u/v,\ldots,u/v}_{(v-1)k},\underbrace{1/v,\ldots,1/v}_{uk} \right) \in (\RR/\ZZ)^{(v+u-1)k}.$$
	Then one has
	$\delta(\Delta,t)=1+t^{k}+t^{2k}+\cdots+t^{(v-1)k}$.
\end{Proposition}

\begin{proof}
Set $\bold{x}=\left(\underbrace{u/v,\ldots,u/v}_{(v-1)k},\underbrace{1/v,\ldots,1/v}_{uk} \right)$ and $\bold{y}=(v/u)\xb =\left(\underbrace{0,\ldots,0}_{(v-1)k},\underbrace{1/u,\ldots,1/u}_{uk} \right)$.
Then we obtain $\hei(\xb)=uk$ and $\hei(\yb)=k$.
Moreover, it follows that
$$\Lambda_\Delta=\{i\xb+j\yb \in (\RR/\ZZ)^{d+1} : i=0,\ldots,v/u-1, \ j=0,\ldots,u-1   \}.$$
%For an integer $n$, $\overline{n}$ denotes the reminder of $n$ devided by $v$.
For any integers $0 \leq i \leq v/u-1$ and $0 \leq j \leq u-1$,
one has 
%\begin{displaymath}
%\begin{aligned}
$$\hei(i \xb + j \yb)=i\hei(\xb)+j\hei(\yb)=(iu+j)k.$$
%=1/v(\overline{(iu'+j)u}(v-1)(k+1)+(iu'+j)u(k+1)\\
%&=1/v(\overline{iv+ju}(v-1)+iv+ju)(k+1)\\
%\end{aligned}
%\end{displaymath}
Hence, it follows from Lemma \ref{delta} that $\delta(\Delta,t)=1+t^k+t^{2k}+\cdots+t^{(v-1)k}$, as desired.
\end{proof}

Furthermore, the following proposition can immediately be obtained from Lemma \ref{delta}.  
\begin{Proposition}
	\label{const}
	Given positive integers $v_1,v_2$ and $k$,
    let $\Delta_1 \subset \RR^{d_1}$ and $\Delta_2 \subset \RR^{d_2}$  be Gorenstein simplices of dimension $d_1$ and $d_2$ such that $\delta(\Delta_1,t)=1+t^{k}+t^{2k}+\cdots+t^{(v_1-1)k}$ and 
	$\delta(\Delta_2,t)=1+t^{v_1k}+t^{2v_1k}+\cdots+t^{v_1(v_2-1)k}$.
	Let $\Delta \subset \RR^{d_1+d_2+1}$ be a lattice simplex of dimension $d_1+d_2+1$ such that 
	$$\Lambda_{\Delta}=\{(\xb,\yb) \in (\RR/\ZZ)^{d_1+d_2+2} : \xb \in \Lambda_{\Delta_1},\yb \in \Lambda_{\Delta_2} \}.$$
	Then one has $\delta(\Delta,t)=1+t^{k}+t^{2k}+\cdots+t^{(v_1v_2-1)k}$.
	In particular, if neither $\Delta_1$ nor $\Delta_2$ is not a lattice pyramid, then $\Delta$ is not a lattice pyramid.
\end{Proposition}
\begin{Remark}
Let $\Delta, \Delta_1$ and $\Delta_2$ be the Gorenstein simplices in Proposition \ref{const}. Then $\Delta$ is the join of $\Delta_1$ and $\Delta_2$.
Hence one has 
\[
\delta(\Delta,t)=\delta(\Delta_1,t)\delta(\Delta_2,t).
\]
\end{Remark}

Now, we consider Problem \ref{pro} for the case that $v$ is $p^2$ or $pq$, where $p$ and $q$ are prime integers with $p \neq q$.  
We see examples of a Gorenstein simplex whose $\delta$-polynomial is $1+t^k+t^{2k}+\cdots+t^{(v-1)k}$ with the explicit vertex representation.

Given a sequence $A=(a_1,\ldots,a_d)$ of integers,
let $\Delta(A) \subset \RR^d$ be the convex hull of the origin of $\RR^d$ and all row vectors of the following matrix:
\begin{displaymath}
\begin{pmatrix}
1 & \ & \ & \ \\
\ & \ddots & \  & \ \\
\ & \ & 1 & \   \\
a_d-a_1 & \cdots & a_d-a_{d-1} & a_d
\end{pmatrix}
\in \ZZ^{d \times d},
\end{displaymath}
where the rest entries are all $0$.
Given sequences $B=(b_1,\ldots,b_s)$ and $C=(c_1,\ldots,c_d)$ of integers with $1\leq s <d$,
let  $\Delta(B,C)\subset \RR^d$ be the convex hull of the origin of $\RR^d$ and all row vectors of the following matrix: 
\begin{displaymath}
\begin{pmatrix}
1 & \ & \ & \ & \ & \ & \ & \ \\
\ & \ddots & \  & \ & \ & \ & \ & \ \\
\ & \ & 1 & \  & \ & \ & \ &  \  \\
b_s-b_1 & \cdots &b_s-b_{s-1} & b_s & \ & \ & \ &\  \\
\ & \ & \ & \ & 1 & \ & \ & \  \\
\ & \ & \ & \ & \ & \ddots & \ & \   \\
\ & \ & \ & \ & \ & \ & 1 & \   \\
c_d-c_1 & \cdots & \cdots & \cdots & \cdots & \cdots &c_d-c_{d-1} & c_d
\end{pmatrix}
\in \ZZ^{d \times d},
\end{displaymath}
where the rest entries are all $0$.

By using Lemma \ref{delta}, we can  compute their $\delta$-polynomials of the lattice simplices in the following propositions.
In fact, it is easy to determine all elements in their associated finite abelian groups.
\begin{Proposition}\label{maincor1}
	Let $p$ be a prime integer and $k$ a positive integer, and we set the following sequences of integers:
	\begin{enumerate}
		\item $A_1=(\underbrace{1,\ldots,1}_{p^2k-2},p^2)$; 
		\item $A_2=\left(\underbrace{1,\ldots,1}_{pk-1}, \underbrace{p ,\ldots,p}_{(p^2-1)k-1},p^2 \right)$;
		\item 
		$B=\left(\underbrace{1,\ldots ,1}_{pk-1},p\right)$, $C=\left(\underbrace{p,\ldots,p}_{pk},\underbrace{1,\ldots,1}_{p^2k-2},p\right)$.
	\end{enumerate} 
	Then the $\delta$-polynomial of each of  $\Delta(A_1)$, $\Delta(A_2)$ and $\Delta(B,C)$ equals $1+t^{k}+t^{2k}+\cdots+t^{(p^2-1)k}$. 
\end{Proposition}
\begin{Proposition}
	\label{maincor2}
	Let  $p$ and $q$ be prime integers with $p \neq q$ and $k$ a positive integer, and we set the following sequences of integers:
	\begin{enumerate}
		\item $A_1=(\underbrace{1,\ldots,1}_{pqk-2},pq)$; 
		\item 
		$B_1=\left(\underbrace{1,\ldots ,1}_{pk-1},p\right)$, $C_1=\left(\underbrace{q,\ldots,q}_{pk},\underbrace{1,\ldots,1}_{pqk-2},q\right)$;
		\item 
		$B_2=\left(\underbrace{1,\ldots ,1}_{qk-1},q\right)$, $C_2=\left(\underbrace{p,\ldots,p}_{qk},\underbrace{1,\ldots,1}_{pqk-2},p\right)$;
		
		\item $A_2=\left(\underbrace{1,\ldots,1}_{pk-1}, \underbrace{p ,\ldots,p}_{(pq-1)k-1},pq \right)$;
		\item $A_3=\left(\underbrace{1,\ldots,1}_{qk-1}, \underbrace{q ,\ldots,q}_{(pq-1)k-1},pq \right)$.	
	\end{enumerate} 
Then the $\delta$-polynomial of each of  
$\Delta(A_1)$, $\Delta(A_2)$, $\Delta(A_3)$, $\Delta(B_1,C_1)$ and $\Delta(B_2,C_2)$
equals $1+t^{k}+t^{2k}+\cdots+t^{(pq-1)k}$.
	
\end{Proposition}

The following theorems are the main results of the present paper.
 \begin{Theorem}\label{main1}
	Let $p$ be a prime integer and $k$ a positive integer, and let $\Delta \subset \RR^d $ be a Gorenstein simplex of dimension $d$ whose $\delta$-polynomial is $1+t^{k}+t^{2k}+\cdots+t^{(p^2-1)k}$. 
Suppose that $\Delta$ is not a lattice pyramid over any lower-dimensional lattice simplex.
Then one of the followings is satisfied:
\begin{enumerate}
	\item $d = p^2k-1$, or
	\item $d=p^2k+(p-1)k-1$, or
	\item $d = p^2k+pk-1$.
\end{enumerate}
Moreover, in each case, a system of generators of the finite abelian group $\Lambda_{\Delta}$ is the set of
row vectors of the matrix which  can be written up to permutation of the columns as follows:
\begin{enumerate}
	\item $(1/p^2 \  \cdots \ 1/p^2) \in (\RR/\ZZ)^{1 \times p^2k}$; 
	\item $\left(\underbrace{1/p \ \cdots \ 1/p}_{(p^2-1)k} \  \underbrace{1/p^2 \   \cdots \ 1/p^2}_{pk} \right) \in (\RR/\ZZ)^{1 \times (p^2+p-1)k}$;
	\item 
	$\begin{pmatrix}
	1/p \ \cdots \  1/p & \ 0 \ \cdots \ \ \  0 \\
	\underbrace{\ \ 0 \ \cdots \  \ \ 0}_{pk} & \underbrace{1/p \ \cdots \ 1/p}_{p^2k}
	\end{pmatrix}
	\in (\RR/\ZZ)^{2 \times p(p+1)k}$.
\end{enumerate} 
In particular, $\Delta$ is unimodularly equivalent to one of $\Delta(A_1)$, $\Delta(A_2)$ and $\Delta(B,C)$ as in Proposition \ref{maincor1}.
\end{Theorem}

\begin{Theorem}\label{main2}
Let $p$ and $q$ be prime integers with $p \neq q$ and $k$ a positive integer, and let $\Delta \subset \RR^d $ be a Gorenstein simplex of dimension $d$ whose $\delta$-polynomial is
$1+t^{k}+t^{2k}+\cdots+t^{(pq-1)k}$.
Suppose that $\Delta$ is not a lattice pyramid over any lower-dimensional lattice simplex.
Then one of the followings is satisfied:
\begin{enumerate}
	\item $d = pqk-1$, or
	\item $d=pqk+pk-1$, or
	\item $d=pqk+qk-1$, or
\item $d=pqk+(p-1)k-1$, or
\item $d=pqk+(q-1)k-1$.
\end{enumerate}
Moreover, in each case, the finite abelian group $\Lambda_{\Delta}$ is generated by one element which can be written up to permutation of the coordinates as follows:
\begin{enumerate}
	\item $(1/(pq), \ldots, 1/(pq)) \in (\RR/\ZZ)^{pqk}$; 
	\item $\left(\underbrace{1/p,\ldots,1/p}_{pk},\underbrace{1/q,\ldots,1/q}_{pqk} \right) \in (\RR/\ZZ)^{p(q+1)k}$;
	\item $\left(\underbrace{1/q,\ldots,1/q}_{qk},\underbrace{1/p,\ldots,1/p}_{pqk} \right) \in (\RR/\ZZ)^{(p+1)qk}$;
	\item $\left(\underbrace{1/q,\ldots,1/q}_{(pq-1)k},\underbrace{1/(pq),\ldots,1/(pq)}_{pk} \right) \in (\RR/\ZZ)^{(pq+p-1)k}$;
	\item $\left(\underbrace{1/p,\ldots,1/p}_{(pq-1)k},\underbrace{1/(pq),\ldots,1/(pq)}_{qk} \right) \in (\RR/\ZZ)^{(pq+q-1)k}$.
\end{enumerate} 
In particular, $\Delta$ is unimodularly equivalent to one of $\Delta(A_1)$, $\Delta(A_2)$, $\Delta(A_3)$, $\Delta(B_1,C_1)$ and $\Delta(B_2,C_2)$ as in Proposition \ref{maincor2}.
\end{Theorem}

\begin{Remark}{\em
The lattice simplices as in Theorems \ref{main1} and \ref{main2} can be constructed by Propositions \ref{prop1} and \ref{const}.}	
\end{Remark}

In order to prove Theorems \ref{main1} and  \ref{main2},
we use the following lemma.
	\begin{Lemma}\label{useful}
		Let $v$ and $k$  positive integers,
		and let $\Delta \subset \RR^d$ be a Gorenstein simplex of dimension d whose $\delta$-polynomial equals 
		$1+t^k+t^{2k}+\cdots+t^{(v-1)k}$.
		Assume that $\xb \in (\RR/\ZZ)^{d+1}$ is an element of $\Lambda_{\Delta}$ such that $\textnormal{ht}(\xb)=k$
		and set $m=\textnormal{ord}(\xb)$.
		Then by reordering the coordinates, we obtain $\xb=\left(\underbrace{1/m,\ldots,1/m}_{s},\underbrace{0,\ldots,0}_{d-s+1} \right) $
		for some positive integer $s$.
	\end{Lemma}

\begin{proof}
Since $m=\textnormal{ord}(\xb)$, $\xb$ must be of a form $\left(k_1/m,\ldots,k_s/m,0,\ldots,0 \right) $ for a positive integer $s$ and integers $1 \leq k_1,\ldots,k_s \leq m-1$ by reordering the coordinates. 
If there exists an integer $k_i \geq 2$ for some $1 \leq i \leq s$, then one has $k_i(m-1)/m \geq 1$. 
Therefore, we obtain $\hei((m-1)\xb) < (m-1)\textnormal{ht}(\xb)=(m-1)k$. 
Since $m=\textnormal{ord}(\xb)$, $(m-1)\xb$ is different from $\bold{0}, \xb,\ldots, (m-2)\xb$. 
We remark that for any $\ab$, $\bb \in (\RR/\ZZ)^{d+1}$, one has $\hei(\ab+\bb) \leq \hei(\ab)+ \hei(\bb)$.
This fact and the supposed $\delta$-polynomial imply that  $\hei(t\xb)=t\hei(\xb)=tk$ for any $1 \leq t \leq m-1$.
This is a contradiction, as desired.
\end{proof}

Finally, we prove  Theorems \ref{main1} and \ref{main2}.

\begin{proof}[Proof of Theorem \ref{main1}]
By Lemma $\ref{p^2}$, $\Delta$ is unimodularly equivalent to either $\Delta_1$ or $\Delta_2$, where $\Delta_1$ and $\Delta_2$ are lattice simplices such that each system of generators of  $\Lambda_{\Delta_{1}}$ and $\Lambda_{\Delta_{2}}$ is the set of vectors of matrix as follows:
\begin{enumerate}
	\item[(i)] 
	$\left(\underbrace{1/p \ \cdots \ 1/p}_{d-s+1} \ \underbrace{1/p^2 \ \ldots \ 1/p^2}_{s} \right) \in (\RR/\ZZ)^{1 \times (d+1)}$;
	\item[(ii)] 
	$\begin{pmatrix}
	(a_0+1)/p & \cdots &  (a_{d-2}+1)/p & 0 & 1/p \\
     (p-a_0)/p  & \cdots & (p-a_{d-2})/p & 1/p & 0 
	\end{pmatrix}
	\in (\RR/\ZZ)^{2 \times (d+1)}$,
\end{enumerate}
where $s$ is a positive integer and  $0 \leq a_0,\ldots ,a_{d-2} \leq p-1$  are integers.

At first, we assume that  $\Delta$ is unimodularly equivalent to $\Delta_1$.
If $s=d+1$, then one has $(d+1)/p^2=k$, hence, $d=p^2k-1$. 
This is the case (1).
Now, we suppose that $s\neq d+1$.
Let $\bold{x}$ be an element of  $\Lambda_{\Delta_{1}}$ with ht$(\bold{x})=k$. 
Then by Lemma \ref{useful}, one has $\bold{x}=\left(\underbrace{0,\ldots,0}_{d-s+1},\underbrace{1/p,\ldots ,1/p}_{s} \right)$,
hence  $s=pk$. 
Set $\bold{y}=\left(\underbrace{1/p,\ldots,1/p}_{d-s+1}, \underbrace{1/p^2,\ldots,1/p^2}_{s} \right)$. Since for any $1 \leq m \leq p-1$, $\text{ht}(m\bold{x})=mk$, we have $\text{ht}(\bold{y})=pk$.
Hence it follows that $d-s+1=p^2k-k$, namely, $d=p^2k+(p-1)k-1$.
This is the case (2).

Next, we assume that $\Delta$ is unimodularly equivalent to $\Delta_2$.
By Lemma \ref{useful}, it follows that for any $0 \leq i \leq d-2$, $a_i \in \left\{ 0,p-1 \right\}$. 
Hence by reordering the coordinates of $\Lambda_{\Delta_{2}}$, we can assume that $\Lambda_{\Delta_{2}}$ is generated by  $$\xb_1=\left(\underbrace{1/p,\ldots,1/p}_{s}, \underbrace{0,\ldots,  0}_{d-s+1} \right), \xb_2=\left(\underbrace{0,\ldots,0}_{s}, \underbrace{1/p,\ldots,1/p}_{d-s+1} \right),$$
where $1 \leq s \leq \lfloor (d+1)/2 \rfloor$. 
Then since $\text{ht}(\xb_1)=k$, one has $s=pk$. 
Moreover, since $\hei(\xb_2)=pk$, we have $d-s+1=p^2k$, namely,
$d=p^2k+pk-1$. 
Therefore, this is the case (3). 

Conversely, in each case, it is easy to show that $\Delta$ is unimodularly equivalent to one of $\Delta(A_1)$, $\Delta(A_2)$ and $\Delta(B,C)$ as in Proposition \ref{maincor1}, as desired.  
\end{proof}

\begin{proof}[Proof of Theorem \ref{main2}]
By Lemma $\ref{pq}$, we can suppose that $\Lambda_{\Delta}$ is generated by $$\bold{x}=\left(\underbrace{1/p,\ldots, 1/p}_{s_1} , \underbrace{1/q,\ldots,1/q}_{s_2}, \underbrace{1/(pq),\ldots,1/(pq)}_{s_3} \right),$$
where $s_1+s_2+s_3=d+1$ with nonnegative integers $s_1,s_2,s_3$.
If $s_1=s_2=0$, since $\hei(\xb)=k$, one has $d = pqk-1$.
This is the case (1).
If $s_3=0$, we can assume that $\Lambda_{\Delta}$ is generated by
$$\xb_1=\left(\underbrace{1/p,\ldots,1/p}_{s_1} , \underbrace{0,\ldots,0}_{s_2} \right), \xb_2=\left(\underbrace{0,\ldots, 0}_{s_1} , \underbrace{1/q,\ldots,1/q}_{s_2} \right),$$
with $s_1,s_2>0$.
Then it follows that  $\hei(\xb_1)=k$ and $\hei(\xb_2)=pk$, or $\hei(\xb_1)=qk$ and $\hei(\xb_2)=k$.
Assume that $\hei(\xb_1)=k$ and $\hei(\xb_2)=pk$. 
Then one has $s_1=pk$ and $s_2=pqk$.
Hence since $d=pqk+pk-1$, this is the case (2). 
Similarly, we can show the case (3).

Next we suppose that $s_1,s_2,s_3>0$.
Let $\ab$ be an element of $\Lambda_{\Delta}$ such that $\hei(\ab)=k$.
By Lemma \ref{useful}, we know that $\ord(\ab)\neq pq$.
Hence, it follows that $\ord(\ab)$ equals $p$ or $q$. 
Now we assume that $\ord(\ab)=p$.
By Lemma \ref{useful} again,  $\bold{a}$ must be of a form $\left(\underbrace{1/p, \ldots,1/p}_{s_1} , \underbrace{0,\ldots,0}_{s_2}, \underbrace{1/p,\ldots, 1/p}_{s_3} \right)$.
Let $\bold{b}=(b_1,\ldots,b_{d+1})$ be an element of $\Lambda_{\Delta}$ such that $\hei(\bold{b})=pk$. 
If there exists an index $1 \leq i \leq s_1$ such that $b_{i}=n/p$ with an integer $1 \leq n \leq p-1$, then $\hei(\bold{b}+(p-1)\bold{a})< \hei (\bold{b})+(p-1) \hei (\bold{a})$. 
Since $\bold{b}+(p-1)\bold{a}$ is different from $\bold{0},\bold{a}, 2\bold{a},\ldots,(p-1)\bold{a},\bold{b},\bold{b}+\bold{a},\ldots,\bold{b}+(p-2)\bold{a}$, this contradicts to that $ \delta_{\Delta}(t)=1+t^k+t^{2k}+\cdots+t^{(pq-1)k}$.
Hence one obtains $b_{i}=0$ for any $1 \leq i \leq s_1$. %and $\ord(\bold{b})=q$, namely 
Therefore, we can assume that $\bold{b}=\left(\underbrace{0,\ldots,0}_{s_1} , \underbrace{\ell/q ,\ldots,\ell/q}_{s_2}, \underbrace{m/q,\ldots,m/q}_{s_3} \right)$ for some positive integers $\ell,m$. 
Then whenever $(g_1,h_1) \neq (g_2,h_2)$ with $0 \leq g_1,g_2 \leq p-1$ and $0 \leq h_1,h_2 \leq q-1$, $g_1\bold{a}+h_1\bold{b}$ and $g_2\bold{a}+h_2\bold{b}$ are different elements of $\Lambda_{\Delta}$.
Hence since $ \delta_{\Delta}(t)=1+t^{k}+t^{2k}+\cdots+t^{(pq-1)k}$, 
one has
$$\hei(g\bold{a}+h\bold{b})=g \hei (\bold{a})+h \hei (\bold{b})$$
for any $0 \leq g \leq p-1$ and $0 \leq h \leq q-1$.
This implies that $\ell=m=1$.
However since $(p-1)/p + (q-1)/q>1$, we have  $\hei((p-1)\bold{a}+(q-1)\bold{b})<(p-1)\textnormal{ht}(\bold{a})+(q-1)\textnormal{ht}(\bold{b})$, a contradiction.
Therefore, it does not follow  $s_1,s_2,s_3>0$.

Finally, we assume that $s_1=0$ and $s_2>0$. 
Then one has $\hei(q\xb)=k$, hence, $s_3=pk$.
Moreover, since $\hei(\xb)=pk$, we obtain $s_2=(pq-1)k$.
Therefore, this is the case (4).
Similarly, we can show the case (5).

Conversely, it is easy to see that $\Delta$ is unimodularly equivalent to one of $\Delta(A_1)$, $\Delta(A_2)$, $\Delta(A_3)$, $\Delta(B_1,C_1)$ and $\Delta(B_2,C_2)$ as in Proposition \ref{maincor2}, as desired.
\end{proof}

\section{The number of Gorenstein simplices}
In \cite[Section $4$]{HT}, we asked how many reflexive polytopes which have the same $\delta$-polynomial exist.
Analogy to this question, in this section, we consider how many Gorenstein simplices which have a given $\delta$-polynomial of Problem \ref{pro} exist.

Given positive integers $v$ and $k$, let $N(v,k)$ denote the number of Gorenstein simplices,
up to unimodular equivalence,
which are not lattice pyramids over any lower-dimensional lattice simplex and whose  $\delta$-polynomials equal $1+t^{k}+t^{2k}+\cdots+t^{(v-1)k}$.
For example,
from Proposition $\ref{moti}$, $N(p,k)=1$ for any prime integer $p$.
Moreover, from Theorems \ref{main1} and \ref{main2},
$N(p^2,k)=3$ and $N(pq,k)=5$ for any distinct prime integers $p$ and $q$.
However, in other case, it is hard to determine $N(v,k)$.
Therefore, our aim of this section is to construct more examples of Gorenstein simplices of Problem \ref{pro}
and to give a lower bound on $N(v,k)$.
 
 The following theorem gives us more examples of Gorenstein simplices of Problem \ref{pro}. 
\begin{Theorem}
	\label{power1}
	Given a positive integer $v$, %and a nonnegative integer $k$, 
	let $\Delta \subset \RR^d$ be a lattice simplex of dimension $d$ such that
	$\Lambda_{\Delta}$ is generated by
	$$\left(\underbrace{1/v_1,\ldots,1/v_1}_{s_1},\underbrace{1/v_2,\ldots,1/v_2}_{s_2}, \ldots, \underbrace{1/v_t,\ldots,1/v_t}_{s_t} \right) \in (\RR/\ZZ)^{d+1},$$
	where $1 < v_1 < \cdots < v_t = v$ and for any $1 \leq i \leq t-1$, $v_i \mid v_{i+1}$ and  $s_{1},\ldots,s_{t}$ are positive integers.
	Then $\delta(\Delta,t)=1+t^{k}+t^{2k}+\cdots+t^{(v-1)k}$
	with a positive integer $k$
	if and only if
	\begin{displaymath}
	s_i=
	\begin{cases}
	\Bigl(\cfrac{v_t}{v_{i-1}}-\cfrac{v_t}{v_{i+1}}\Bigr)k, &\ 1 \leq i \leq t-1\\
	\cfrac{v_t}{v_{t-1}}k, &\ i=t,
	\end{cases}
	\end{displaymath}
	where $v_{0}=1$.
\end{Theorem}

\begin{proof}
Let  $$\xb_0=\left(\underbrace{1/v_1,\ldots,1/v_1}_{s_1},\underbrace{1/v_2,\ldots,1/v_2}_{s_2}, \ldots, \underbrace{1/v_t,\ldots,1/v_t}_{s_t} \right) \in (\RR/\ZZ)^{d+1},$$
and for $i=1,\ldots,t-1$, we set $\xb_i=v_i\xb_0$.
Then  it follows that
$$\Lambda_\Delta=\left\{\sum\limits_{i=0}^{t-1}c_i\xb_i : c_i \in \ZZ_{\geq 0}, 0 \leq c_i \leq v_{i+1}/v_i-1  \text{\ for \ } i=0,\ldots,t-1 \right\}.$$
Moreover, we obtain $\hei(\xb_{i})=\sum_{j=1}^{t-i} \dfrac{v_{i}}{v_{i+j}}s_{{i+j}}$ for  $i=0, \ldots, t-1$.
Since $$\hei(\xb_{i})=\hei\Bigl(\dfrac{v_i}{v_{i-1}} \xb_{i-1}\Bigr)=\dfrac{v_i}{v_{i-1}}\hei(\xb_{i-1})-s_i$$ for any $1 \leq i \leq t-1$,
it follows that
for any $1 \leq i \leq t-1$, $s_i=\Bigl(\dfrac{v_t}{v_{i-1}}-\dfrac{v_t}{v_{i+1}}\Bigr)k$ and $s_t=\dfrac{v_t}{v_{t-1}}k$
if and only if for any $0 \leq i \leq t-1$, $\hei(\xb_i)=\dfrac{v_t}{v_{i+1}}k$.
Hence we should prove that $\delta(\Delta,t)=1+t^{k}+t^{2k}+\cdots+t^{(v-1)k}$ if and only if 
for any $0 \leq i \leq t-1$, $\hei(\xb_i)=\dfrac{v_t}{v_{i+1}}k$.

At first, we assume that $\delta(\Delta,t)=1+t^{k}+t^{2k}+\cdots+t^{(v-1)k}$.
By Lemma \ref{useful},
one has $\hei(\xb_{t-1})=k$. 
Suppose that for any $n \leq i \leq t-1$, $\hei(\xb_{i})=\dfrac{v_{t}}{v_{i+1}}k$ with an integer $1 \leq n \leq t-1$.
Then since $\hei(\sum_{i=n}^{t-1}(v_{i+1}/v_i-1)\xb_i)=(v_t/v_n-1)k$, there exists an integer $m$ with $0 \leq m \leq n-1$ such that $\hei(\xb_{m})=\dfrac{v_t}{v_{n}}k$.
Now, we assume that $m < n-1$. 
Set
$$ \Lambda'=\left\{ 
c_{m}\bold{x}_{m}+\sum\limits_{i=n}^{t-1}c_i\xb_i : 0 \leq c_i \leq v_{i+1}/v_i-1   \text{\ for \ } i=m,n,n+1,\ldots,t-1\right\}.$$
Then one has $\{\hei(\bold{x}) : \xb \in \Lambda' \}=\{ jk : j=0,\ldots,(v_{m+1}v_{t})/(v_{m}v_{n})-1\}.$
However,
$$\hei(\xb_{m+1})=\hei\Bigl(\cfrac{v_{m+1}}{v_{m}}\xb_{m}\Bigr)<\cfrac{v_{m+1}}{v_{m}}\hei(\xb_{m})=\Bigl(\cfrac{v_{m+1}v_t}{v_{m}v_n}\Bigr)k.$$
and $\xb_{m+1}$ is not in $\Lambda'$, a contradiction. 
Hence we obtain $\hei(\xb_{i-1})=\dfrac{v_t}{v_{i}}k$ for any $0 \leq i \leq t-1$.

Conversely,  we assume that for any $0 \leq i \leq t-1$, $\hei(\xb_i)=\dfrac{v_t}{v_{i+1}}k$.
Since for any $c_i$ with $0 \leq c_i \leq v_{i+1}/v_i-1$, $\hei (\sum_{i=0}^{t-1}c_i\xb_i)=\sum_{i=0}^{t-1}c_i\hei(\xb_i)$,
one has $\delta(\Delta,t)=1+t^{k}+t^{2k}+\cdots+t^{(v-1)k}$,
as desired.
\end{proof}

By Theorems \ref{power1} and \cite[Theorem 2.2]{Tsuchiya}, we can answer  Problem \ref{pro}
when $v$ is a power of a prime integer and the associated finite abelian group is cyclic, namely,
it is generated by one element.
\begin{Corollary}
	\label{power2}
Let $p$ be a prime integer, $\ell$ and $k$ a positive integers,
and 
	let $\Delta \subset \RR^d$ be a lattice simplex of dimension $d$ such that $\Lambda_{\Delta}$ is cyclic and 
$\delta(\Delta,t)=1+t^{k}+t^{2k}+\cdots+t^{(p^{\ell}-1)k}$.
Suppose that $\Delta$ is not a lattice pyramid over any lower-dimensional lattice simplex.
Then there exist positive integers $0 < \ell_1 < \cdots < \ell_t = \ell$ and $s_{1},\ldots,s_{t}$
such that the following conditions are satisfied:
\begin{itemize}
	
	\item $\Lambda_{\Delta}$ is generated by
	$$\left(\underbrace{1/p^{\ell_1},\ldots,1/p^{\ell_1}}_{s_{1}},\underbrace{1/p^{\ell_2},\ldots,1/p^{\ell_2}}_{s_{2}}, \ldots, \underbrace{1/p^{\ell_t},\ldots,1/p^{\ell_t}}_{s_{t}} \right) \in (\RR/\ZZ)^{d+1}$$
		for some ordering of the vertices of $\Delta$;
		\item It follows that
	\begin{displaymath}
	s_i=
	\begin{cases}
	(p^{\ell-\ell_{i-1}}-p^{\ell-\ell_{i+1}})k, &\ 1 \leq i \leq t-1\\
	 p^{\ell-\ell_{t-1}}k, &\ i=t,
	\end{cases}
	\end{displaymath}
	where $\ell_{0}=0$.
\end{itemize}
\end{Corollary}

Now, we consider to give a lower bound on $N(v,k)$.
Given positive integers $v$ and  $k$, let $M(v,k)$ denote the number of Gorenstein simplices, up to unimodular equivalence,
which appeared in Theorem \ref{power1}.
Then one has $N(v,k) \geq M(v,k)$.
By Theorem \ref{power1}, we can determine $M(v,k)$ in terms of the divisor lattice of $v$.
Given a positive integer $v$, let $D_v$ the set of all divisors of $v$, ordered by divisibility.
Then $D_v$ is a partially ordered set, in particular, a lattice, called the \textit{divisor lattice} of $v$.
We call subset $C \subset D_v$ a {\em chain} of $D_v$ if $C$ is a totally ordered subset with respect to the induced order. 
\begin{Corollary}
\label{cor:chain}
	Let $v$ and $k$ be positive integers.
	Then $M(v,k)$ equals the number of chains from a non-least element to the greatest element in $D_v$. 
	In particular, one has $M(v,k)=\sum_{n \in D_v\setminus \{v\}}M(n,k)$.
\end{Corollary}
We give examples of $M(v,k)$.
\begin{Example}\label{ex}
	{\em
	(1) Let $v=p^\ell$ with a prime integer $p$ and a positive integer $\ell$.
	Then from Corollary \ref{cor:chain}, we know that $M(v,k)$ equals 
	the number of subsets of $\{1,\ldots,\ell-1\}$.
	Hence one has $M(v,k)=2^{\ell-1}$.
	
	(2) Let $v=p_1\cdots p_t$, where $p_1,\ldots,p_t$ are distinct prime integers. 
	From Corollary \ref{cor:chain}, we know that $M(v,k)$ depends only on $t$. 
	Now, let $a(t)=M(v,k)$, where we define $a(0)=M(1,k)=1$. 
	Then one has 
	$$ a(t)=M(v,k)=\sum\limits_{n \in D_v\setminus \{v\}}M(n,k)=1+\sum\limits_{i=1}^{t-1}\binom{t}{i} M(p_1\cdots p_i,k)=\sum\limits_{i=0}^{t-1}\binom{t}{i}a(i). $$
 We remark that $a(t)$ is the well-known recursive sequence (\cite[A000670]{Bell}) which is called  the {\em  ordered Bell numbers} or {\em Fubini numbers}.
}
\end{Example}

Corollary \ref{cor:chain} says that $M(v,k)$ depends only on the divisor lattice $D_v$.
In particular, letting $v=p_1^{a_1} \cdots p_{t}^{a_t}$ with distinct prime integers $p_1,\ldots,p_t$ and positive integers $a_1,\ldots,a_t$, $M(v,k)$ depends only on $(a_1,\ldots,a_t)$.
On the other hand, $N(v,k)$ depends only on the divisor lattice $D_v$ when $v$ is a prime integer or the product of two prime integers.
Therefore, we conjecture the following:
\begin{Conjecture}
	Let $v$ and $k$ be positive integers. Then $N(v,k)$ depends only on the divisor lattice $D_v$ of $v$.  
\end{Conjecture}

Let $\Pc$ be a lattice polytope whose $\delta$-polynomial is 
	$1+t^{k}+\cdots+t^{(v-1)k}$ with some positive integers $v$ and $k$.
	If $k > 1$, then $\Pc$ is a simplex. 
	However, if $k=1$, then $\Pc$ is not always a simplex.
	We see the list of the Gorenstein non-simplices whose $\delta$-polynomials are $1+t+\cdots+t^{v-1}$ when $2 \leq v \leq 4$.
	
	\begin{Proposition}[\cite{HTclass}]
		Let $v$ be a positive integer with $2 \leq v \leq 4$ and $\Pc$ a Gorenstein non-simplex whose $\delta$-polynomial is $1+t+\cdots+t^{v-1}$.
		Suppose that $\Pc$ is not a lattice pyramid over any lower-dimensional lattice polytope.
		\begin{enumerate}
		\item[(a)] When $v=2$, $\Pc$ is unimodularly equivalent to a unit square.
		\item[(b)] When $v=3$, $\Pc$ is unimodularly equivalent to the lattice polytope which is the convex hull of
		\begin{enumerate}
		\item[(1)] ${\bf 0},\eb_1,\eb_2,\eb_3,\eb_1+\eb_2-2\eb_3 \in \RR^3$, or
		\item[(2)] ${\bf 0},\eb_1,\eb_2,\eb_3,\eb_4,-\eb_1-\eb_2+\eb_3+\eb_4 \in \RR^4$.
		\end{enumerate}
		\item[(c)] When $v=4$, $\Pc$ is unimodularly equivalent to the lattice polytope which is the convex hull of
		\begin{enumerate}
				\item[(1)] ${\bf 0}, \eb_1,\eb_2,\eb_3, \eb_4,-\eb_1-\eb_2-\eb_3+\eb_4 \in \RR^4$, or
		\item[(2)] ${\bf 0}, \eb_1,\eb_2,\eb_3, \eb_4,-\eb_1-\eb_2-\eb_3+2\eb_4 \in \RR^4$, or
		\item[(3)] ${\bf 0}, \eb_1,\eb_2,\eb_3, \eb_4,\eb_5,-2\eb_1-\eb_2+\eb_3+\eb_4+\eb_5 \in \RR^5$, or
			\item[(4)] ${\bf 0}, \eb_1,\eb_2,\eb_3, \eb_4,\eb_5,\eb_6,-\eb_1-\eb_2-\eb_3+\eb_4+\eb_5+\eb_6 \in \RR^6$, or
			\item[(5)] ${\mathbf 0}, \eb_1,\eb_2,\eb_{3},\eb_2+\eb_3+2\eb_4,-\eb_{1}+\eb_{2} \in \RR^4$, or
			\item[(6)] ${\mathbf 0}, \eb_1,\eb_2,\eb_3,\eb_{4},\eb_3+\eb_4+2\eb_5,\eb_{1}+\eb_{2} \in \RR^5$, or
			\item[(7)]  ${\mathbf 0}, \eb_1,\eb_2,\eb_3,\eb_4,\eb_{5},\eb_4+\eb_5+2\eb_6,\eb_{1}+\eb_{2}-\eb_{3} \in \RR^6$.
		\end{enumerate}
				\end{enumerate}
				Here ${\bf 0}$ is the origin of $\RR^d$ and $\eb_1,\ldots,\eb_d$ are the canonical unit coordinate vectors of $\RR^d$.
	\end{Proposition}
Therefore, the number of Gorenstein polytopes up to unimodular equivalence, which are not lattice pyramids over any lower-dimensional lattice polytope and whose $\delta$-polynomials equal $1+t+\cdots+t^{v-1}$ with some positive integer $v$ does not depend only on the divisor lattice $D_v$ of $v$.

\end{document}